\documentclass{amsart}                     % onecolumn (standard format)

\theoremstyle{theorem}
\newtheorem{theorem}{Theorem}
\theoremstyle{lemma}
\newtheorem{lemma}{Lemma}
\theoremstyle{definition}
\newtheorem{definition}{Definition}
\newtheorem*{remark}{Remark}
\newtheorem{proposition}{Proposition}
\newtheorem{corollary}{Corollary}

\usepackage[normalem]{ulem}
 \usepackage{graphicx}

\usepackage{epsfig}
\usepackage{epstopdf}
\usepackage{subfigure}
\RequirePackage{fix-cm}
\usepackage{mathptmx}      % use Times fonts if available on your TeX system
\usepackage{amsmath}
\usepackage{amssymb}

\numberwithin{equation}{section}

\usepackage[normalem]{ulem}
\usepackage{graphicx}

\usepackage{epsfig}
\usepackage{epstopdf}
\usepackage{subfigure}

\usepackage{tabularx}
\usepackage{array}
\usepackage{dcolumn}

\begin{document}

\title[Extending a Function Just by Multiplying and Dividing Function Values]{Extending a Function Just by Multiplying and Dividing Function Values:\ \ Smoothness and Prime Identities}



\author[P. A. Miller]{Patrick Arthur Miller \\   \\ AT\&T Laboratories (retired) \\ Middletown, New Jersey, USA \\ \lowercase{email:\ \  patrickarthurmiller@gmail.com}}
%\address{Bell Telephone Laboratories / Bell Communications Research / AT\&T Labs (``retired'') \\}
%\email{patrickarthurmiller@gmail.com}

\begin{abstract}
We describe a purely-multiplicative method for extending an analytic function.  It calculates the value of an analytic function at a point, merely by multiplying together function values and reciprocals of function values at other points closer to the origin.  The function values are taken at the points of geometric sequences,  independent of the function, whose geometric ratios are arbitrary.  The  method exposes an ``elastic invariance'' property of all analytic functions.  We show how to simplify and truncate multiplicative function extensions for practical calculations.   If we choose each geometric ratio to be the reciprocal of a power of a prime number, we obtain a prime functional identity, which contains a generalization of the M\"obius function (with the same denominator as the Riemann zeta function), and generates prime number identities.
\end{abstract}


\maketitle

%\ \newpage
%\textbf{A Purely-Multiplicative Function Extension Method}\ \ \\
%{\centering{Patrick Arthur Miller}\par}

%%%%% END OF TITLE PAGE %%%%%%%%%%%%%%%%%%%%%%%%%%%%%%%%%%%%%%%%%%%%%%\

\setcounter{tocdepth}{1}
\tableofcontents

\section{Introduction}

Taylor Series and the Cauchy Integral Theorem extend an analytic function outward from a point or inward from a curve, respectively, using differentiation or integration.  We describe a purely-multiplicative method for extending an analytic function.  Our method does not use caculus. It calculates the value of an analytic function at a point in the complex plane, merely by multiplying  together function values and reciprocals of function values, obtained at other points closer to the origin.  The function values are obtained at the points of geometric sequences, independent of the function, whose geometric ratios are arbitrary.

The  method exposes a hitherto-unknown invariance property of all analytic functions.   If we vary the geometric ratios, the evaluation points of the analytic function move, but the product of function values and reciprocals of function values is invariant.  We call this property of analytic functions ``elastic invariance''.  A non-smooth function cannot have elastic invariance, because the moving function-evaluation points would cross any discontinuities of the function or its derivatives, causing the product of function values and reciprocals of function values to vary.  ``Has elastic invariance'' could serve to define smoothness for a function or curve without the use of calculus.

If we choose a common geometric ratio for every geometric sequence, we can simplify practical calculation of truncated multiplicative function extensions.  We describe other results, such as truncation error, isolating a component factor of a function, and expressing the derivative of a function in terms of an infinite product of function values .

If, on the other hand, we choose each geometric ratio to be distinct, namely, a power of the reciprocal of a distinct  prime number, we obtain a prime functional identity for each analytic function.  The prime functional identity contains a generalization of the M\"obius function (with the same denominator $1/n^s$ as the Riemann zeta function), and generates prime number identities.

\section{Multiplicative Function Extension}

\begin{theorem}[Multiplicative Function Extension]
\label{thm:main}
 Suppose that $U \subset \mathbb{C}$ is an open disk-shaped neighborhood of the origin and that $z \in U$.  Suppose that the function $f:\mathbb{C} \rightarrow \mathbb{C}$ is analytic and nonzero in $U$ and that $f(0)=1$.\ \ Let $ F_{\mathbb{N},\, max=m}^* =\ \{ S\,| ,\, S \subset \mathbb{ N},\ S \neq \emptyset ,\ \max(S)=m\}$ be the family of nonempty subsets of the positive integers whose maximum element is $m$.   For every integer $k\geq 1$, let  $r_k \in \mathbb{C}$ with $|r_k|>1$\, and\ \ $\Re\left((r_k)^k\right) \geq 1/2$\,.\ \  Then
\begin{equation}
\label{eq:quotient}
f(z)\, = \prod_{m=1}^\infty \,\,\  \prod_{S \in  F_{\mathbb{N},\,\, max=m}^*} \,\,\,  \prod_{\substack{(\forall k \in  S) \\
n_k = 1
}}^\infty\  \left[ f\left(\left[\prod\limits_{k\in S}\frac{[(r_k)^k-1]^{1/k}}{(r_k)^{n_k}}\right] \cdot z \right)\right]^{\big{(-1) ^{\,| S|-1 }}}\,.
\end{equation}
\end{theorem}

\begin{remark}
More explicitly, the theorem states that
\begin{align}
\label{eq:explicit}
f(z)\ &=\ \prod_{n_1 = 1}^\infty f\left( \frac{r_1-1}{(r_1)^{n_1}} \cdot z \right)\ \times\ \frac{ \prod\limits_{n_2 = 1}^\infty f\left( \frac{\left[(r_2)^2-1\right]^{1/2}}{(r_2)^{n_2}} \cdot z \right)}{ \prod\limits_{n_1,n_2 = 1}^\infty f\left( \frac{(r_1-1)\,\left[(r_2)^2-1\right]^{1/2}}{(r_1)^{n_1}\,(r_2)^{n_2}} \cdot z \right)} \nonumber \ \\
&\ \ \ \ \ \ \ \ \ \ \ \ \ \times\ \frac{ \prod\limits_{n_3 = 1}^\infty f\left( \frac{\left[(r_3)^3-1\right]^{1/3}}{(r_3)^{n_3}} \cdot z \right)\,  \prod\limits_{n_1,n_2,n_3 = 1}^\infty f\left( \frac{(r_1-1)\,\left[(r_2)^2-1\right]^{1/2}\,\left[(r_3)^3-1\right]^{1/3}}{(r_1)^{n_1}\,(r_2)^{n_2}\,(r_3)^{n_3}} \cdot z \right)}{ \prod\limits_{n_1,n_3 = 1}^\infty f\left( \frac{(r_1-1)\,\left[(r_3)^3-1\right]^{1/3}}{(r_1)^{n_1}\,(r_3)^{n_3}} \cdot z \right)\, \prod\limits_{n_2,n_3 = 1}^\infty f\left( \frac{[(r_2)^2-1]^{1/2}\,\left[(r_3)^3-1\right]^{1/3}}{(r_2)^{n_2}\,(r_3)^{n_3}} \cdot z \right)}\times\cdots 
\end{align}
\end{remark}

\section{Simple Notation for Multiplicative Function Extension}
To make Theorem~\ref{thm:main} more easily understandable and easier to prove, we decompose it into several parts using the following definitions.

\begin{definition}  We define a geometric sequence $X$ for a \textbf{single positive integer}.\ \ For the positive interer $k$,\ let $r_k \in \mathbb{C}$ with $1 < |r_k| < \infty$.\ \ We define the \uline{geometric sampling sequence}
generated by the positive-integer $k$,  with index $n_k$ and geometric ratio $1/ r_k$,\ \ to be\,\ $\{X_{k,n_k}\}_{n_k=1}^\infty$\,,\, where 
\begin{equation*}
\label{eq:places}
X_{k,n_k}\, =\, \left[\left(r_k\right)^k-1\right]^{1/k}\,\cdot\frac{1}{(r_k)^{n_k}}\,\,.
\end{equation*}
\end{definition}
\begin{definition}  We define a product $P$ for a \textbf{set of positive integers}.  We define $P(f,\, S,\,z)$, the \uline{geometric sampling product} for the function $f: \mathbb{C} \rightarrow \mathbb{C}$, the positive-integer set $S\subset \mathbb{N}$, and the point $z$ to be
\begin{equation}
\label{eq:manyfactors}
 P(f,\, S,\,z)\,\,=\,  \prod_{\substack{
(\forall k \in  S) \\
n_k = 1
}}^\infty \, f\left( \left[\prod_{k\in  S}\ X_{k,n_k}\right] z\right)\,.
\end{equation}
\end{definition}
\begin{remark}  To calculate $P$, we first form function evaluation points by multiplying $z$ by one value from each of the geometric sampling sequences whose integer $k$ is in $S$.  We evaluate the function at all such points and multiply all of the function values.
\end{remark}

\begin{definition}  We define a qotient $Q$ for a \textbf{family of sets of positive integers}.  Suppose that $F$ is a family of sets, consisting of any finite subsets of the positive integers $\mathbb{ N}$.  We define the \uline{geometric sampling quotient} for the nonzero function $f: \mathbb{C} \rightarrow \mathbb{C}$ with $f(0)=1$, the family of finite positive-integer sets $F$, and the point $z$ to be
\begin{equation}
Q\left(f,\, F,\,z\right) \,=\,  \frac{\ \ \prod\limits_{ S \in  F,\,\,| S|\, \textrm{is odd}}\,\, P\left(f,\, S,\,z\right)\ } {\,\prod\limits_{ S \in  F,\,\,| S|\, \textrm{is even}}\, P\left(f,\, S,\,z\right)}\,. \label{eq:Q}
\end{equation}
\end{definition}
\begin{remark}  The exponent $(-1)^{|S|-1}$ in Eq.~\ref{eq:quotient} places the function value in the numerator (denominator) if the set $S$ contains an odd (even) number of integers, as in Eq.~\ref{eq:Q}.
\end{remark}

\begin{remark}  Equation~\ref{eq:quotient} (Theorem~\ref{thm:main}) can be expressed very simply in terms of $Q$ and $P$:
\begin{align}
f(z) &=\  Q(f,\,F_{\mathbb{N}}^*,\,z)\ _{(GEO)} 
\ \nonumber \\
&=\ P(f,\,\{1\},\,z)  \times  \frac{P(f,\,\{2\},\,z)}{P,\,(f,\,\{1,2\},\,z)} \times \frac{P(f,\,\{3\},\,z) P(f,\,\{1,2,3\},\,z)}{P(f,\,\{1,3\},\,z) P(f,\,\{2,3\},\,z)} \times \cdots \label{eq:egGEO1} 
\end{align}
 (GEO) indicates that the geometric sampling products are to be multiplied and divided in the Greatest Element Order of their integer subsets, as described in Definition~\ref{def:GEO}, below.
 \end{remark}

\begin{definition} \label{def:GEO}
\uline{Greatest Element Order} (GEO) is an ordering of  families of nonempty subsets of $\mathbb{N}$ in which
\begin{enumerate}
\item The first subset is \{1\} (if present)\,,
\item The next group of subsets are those whose greatest element is 2:\  \{2\} and \{1, 2\} (if present)\,,
\item The next group of subsets are those whose greatest element is 3: \ \{3\}, \{1, 2, 3\}, \{1, 3\}, and \{2, 3\} (if present)\,,
\item etc.
\end{enumerate}
\end{definition}

\begin{remark}
Greatest Element Order is illustrated in Eq.~\ref{eq:egGEO1}\,.  Note that Greatest Element Order partitions the nonempty subsets of $\mathbb{N}$ into equivalence classes with the equivalence relation ``have the same greatest element".  Greatest Element Ordering orders these equivalent classes, but does not specify the position of a particular subset within its equivalence class.  The index $m$ in Theorem~\ref{thm:main} labels these equivalence classes.
\end{remark}

\section{Preliminary Lemmas and Definitions}

To prove Theorem \ref{thm:main}, we need the following lemmas and definitions.  They describe the many useful properties of the product $P$ and quotient $Q$, and define useful integer sets.\newline

Theorem~\ref{thm:main} applies only to analytic functions.  If the complex function $f$ is analytic and nonzero in a neighborhood of the origin, therein it has an analytic logarithm that has a power series \cite{Conway}, and, since $f(0)=1$, can be written as
\begin{equation}
\label{eq:taylog}
f(z) = e^{\ln(f)} = \exp\left(\sum_{k=1}^\infty c_k\,z^k\right)= \prod_{k=1}^\infty \exp\left(c_k\,z^k\right) = \prod_{k=1}^\infty\exp \left[\frac{1}{k!}\left(\frac{d^k \ln f(z)}{dz^k}\right)_{z=0}z^k \right]\,,
\end{equation}
where we have used Taylor Series to decompose the function $\ln f$ into \uline{terms} of an infinite sum.
The following definition formalizes the decomposition of the function $f$ into \uline{factors} of an infnite product.
\begin{definition}
Assume that $f$ is analytic in a neighborhood of the origin and that $f(0)=1$.  We call
\begin{equation}
f_k(z)=\exp(c_k\,z^k),\ \  \textrm{where}\ \ c_k = \frac{1}{k!}\left(\frac{d^k \ln f(z)}{dz^k}\right)_{z=0}\,,
\end{equation}
 the \uline{$k^{th}$ function factor of $f$},\ because $f(z) = \prod\limits_{k=1}^\infty\ f_k(z)$.
\end{definition}
\begin{remark}
When $f(0)=1$, $c_0=0$, and so $f_0=1$ (a factor of 1) can be omitted.
\end{remark}
In what follows
\begin{itemize}
\item We assume the functions $f$ and $g$ satisfy the requirements of Theorem~\ref{thm:main}.
\item  An asterisk (*) attached as a superscript to a family of sets, as in $F^*$, indicates that the empty set $\emptyset$ is not a member of the family of sets.
\item For simplicity we shall often omit the argument $z$ from the product $P$ and quotient $Q$, to focus on their other arguments:  the function $f$ and the integer set $S$ or family of sets $F$.
\end{itemize}

The first lemma describes an invariance, key to the proof of Theorem~\ref{thm:main}.
\begin{lemma} \label{lem:Invar}
\item The geometric sampling product for the function factor $f_k$ and the set $\{k\}$ is the function $f_k$ itself:
\begin{equation}
\label{eq:Invar}
 P(f_k,\,\{k\})\, =\, f_k\,.
\end{equation}
\end{lemma}
\begin{proof}\ \ The proof is elmentary algebra:
\begin{align*}
 P(f_k,\,\{k\},\,z)\, &=  \prod_{n_k=1}^\infty\, f_k\left(X_{k,n_k}\cdot z\right) \\ &=\  \prod_{n_k=1}^\infty\, f_k\left(\frac{\left[(r_k)^k-1)^{1/k}\right]\, z}{(r_k)^{n_k}}\right) \nonumber \\ 
&= \prod_{n_k=1}^\infty\, \textrm{exp}\left[c_k \left(\frac{\left[(r^k)-1\right]^{1/k}\, z}{(r_k)^{n_k}}\right)^k\right] \nonumber \\ 
&=\  \exp\left[c_kz^k\left[(r_k)^k-1\right]\sum_{n_k=1}^\infty\, \left(\frac{1}{(r_k)^k}\right)^{n_k}\right] \nonumber  \\
&=\  \exp\left[c_kz^k\left[(r_k)^k-1\right]\frac{1/r^k}{1-1/r^k}\right] \nonumber \\
&=\  \exp(c_kz^k) \ =\  f_k(z)\,.  \nonumber
\end{align*}
\end{proof}
\begin{lemma} \label{lem:erasek}
Let $S$ be any finite subset of $  \mathbb{N}$.  Then
\begin{equation}
 P(f_k,\, S) =  P(f_k,\, S \backslash\,\{k\})\,. \label{eq:erasek}
\end{equation}
\end{lemma}
\begin{remark}  In the product $ P(f_k,\,S)$, the component $f_k$ effectively ``removes'' any $k$ in $S$.
\end{remark}
\begin{proof}
This is obviously true for any $ S$ that does not contain $k$, since then $ S \backslash \{k\} =  S$.  To prove it when the set $S$ does contain the integer $k$, we separate the $k$-th coefficient of $z$ from the others, and then apply Eq.~\ref{eq:Invar} to $z'=\ $[non-$k$ coefficients]$\cdot z$.
\begin{align*}
P(f_k,\,S,z) &= \prod_{\substack{(\forall j \in  S) \\
n_j = 1
}}^\infty f_k\left(\left[\prod_{j\in S}X_{j,n_j}\right]z\right) \\
&=  \prod_{\substack{(\forall j \in  S\backslash k) \\
n_j = 1
}}^\infty\,\,\prod_{n_k=1}^\infty f_k\left(X_{k,n_k}\cdot \left[\prod_{j\in S \backslash \{k\}}X_{j,n_j}\right]z\right) \\
&=  \prod_{\substack{(\forall j \in  S\backslash k) \\
n_j = 1
}}^\infty\,\,P\left( f_k,\,\{k\},\,\left[\ \prod_{j\in S \backslash \{k\}}X_{j,n_j}\right]z\right) \\
&=  \prod_{\substack{(\forall j \in  S\backslash k) \\
n_j = 1
}}^\infty f_k\left( \left[\prod_{j\in S \backslash \{k\}}X_{j,n_j}\right]z\right) \\
&=  P(f_k,\,S \backslash \{k\},\, z)\,.
\end{align*}
\end{proof}  

\begin{lemma} \label{lem:emptysetf}
For consistency, we must have $P(f_k,\,\emptyset) = f_k$.
\end{lemma}
\begin{proof}
\begin{equation}
\label{eq:emptysetf}
f_k = P(f_k,\,\{k\}) = P(f_k,\,\{k\}\backslash \{k\}) = P(f_k,\,\emptyset)\,.
\end{equation}
\end{proof}  

\begin{lemma} \label{lem:prodf}
If $f$ and $g$ are functions of $z$, then
\begin{equation} \label{eq:prodf}
P(f\cdot g,\, S) = P(f,\, S) \cdot P(g,\, S)\,.
\end{equation}
\end{lemma}
\begin{proof}
This is true because $P$ is a product of function values and multiplication is commutative.
\end{proof}  

\begin{lemma} \label{femptyset}
For any analytic function$f$, $P(f,\,\emptyset) = f$.
\end{lemma}
\begin{proof}
Lemma~\ref{lem:prodf} implies
\begin{equation}
\label{eq:femptyset}
P(f,\emptyset)\ =\ P(\prod_{k=1}^\infty f_k,\,\emptyset)\ =\ \prod_{k=1}^\infty P(f_k, \emptyset)\ =\ \prod_{k=1}^\infty f_k\ =\ f\,.
\end{equation}
\end{proof}  

\begin{definition}
For any family $F$ of finite sets, we define
\begin{align*}
 F^{\,\text{odd}} & = \{ S\,|\, S \in  F,\,| S|\,\text{\,is\,odd}\}\,, \\
 F^{\,\text{even}} & = \{ S\,|\, S \in  F,\,| S|\,\text{\,is\,even}\}\,.
\end{align*}
\end{definition}

\begin{lemma} \label{lem:Qprod}
\begin{equation} \label{eq:Qprod}
Q\left(f \cdot g\,,\, F\right) = Q\left(f,\,F\right) \cdot Q\left(g,\,F\right).
\end{equation}
\end{lemma}
\begin{proof}
This follows from Lemma~\ref{lem:prodf}.
\begin{equation}
Q(f \cdot g, F)\ =\ \frac{P(f \cdot g, F^{\,\text{odd}})}{P(f \cdot g, F^{\,\text{even}})}\ =\ \frac{P(f, F^{\,\text{odd}}) \cdot P(g, F^{\,\text{odd}})}{P(f, F^{\,\text{even}}) \cdot P(g, F^{\,\text{even}})}\ =\  Q\left(f,\,F\right) \cdot Q\left(g,\,F\right)\,.
\end{equation}
\end{proof}

%\begin{lemma} \label{lem:Punion}
%If $F_1$ and $F_2$ are two families of sets of positive integers, and if $F_1 \cap F_2 = \emptyset$,
%\begin{equation} \label{eq:Qunion}
%P\left(f,\,F_1 \cup F_2\right) = P\left(f,\,F_1\right) \cdot P\left(f,\,F_2\right)\,.
%\end{equation}
%\end{lemma}
%\begin{proof}
%This follows from the fact that  $P$ is a product over sets.
%\end{proof}  

\begin{lemma} \label{lem:Qunion}
If $F_1$ and $F_2$ are two families of sets of positive integers, and if $F_1 \cap F_2 = \emptyset$,
\begin{equation} \label{eq:Qunion}
Q\left(f,\,F_1 \cup F_2\right) = Q\left(f,\,F_1\right) \cdot Q\left(f,\,F_2\right)\,.
\end{equation}
\end{lemma}
\begin{proof}
This follows from the definition of $Q$.
\begin{equation}
Q\left(f,\,F_1 \cup F_2\right)\ =\ \frac{Q(f, F_1^{\,\text{odd}}) \cdot Q(f, F_2^{\,\text{odd}})}{Q(f, F_1^{\,\text{even}}) \cdot Q(f, F_2^{\,\text{even}})}\ =\ Q\left(f,\,F_1\right) \cdot Q\left(f,\,F_2\right)\,.
\end{equation}
\end{proof}  

\begin{lemma} \label{lem:emptyset}
Let $F_\mathbb{N}$ be the family of all subsets of $\mathbb{N}$ including $\emptyset$, and let $F_\mathbb{N}^*$ be the family of all subsets of $\mathbb{N}$ excluding $\emptyset$.  Then
\begin{equation}
Q(f\,, F_\mathbb{N}) = \frac{Q(f,\,F_ \mathbb{N}^*)}{f}. \label{eq:emptyset}
\end{equation}
\end{lemma}
\begin{proof}
This follows from $F_\mathbb{N} = F_\mathbb{N}^* \cup \emptyset$\,, Lemma~\ref{lem:Qunion}, and the fact that $| \emptyset | =0$ is even.
\begin{equation}
Q(f,\,F_\mathbb{N})\ =\ Q(f,\,F_\mathbb{N}^*) Q(f,\,\emptyset)\ =\ \frac{Q(f,\,F_\mathbb{N}^*)}{P(f,\, \emptyset)}\ =\ \frac{Q(f,\,F_\mathbb{N}^*)}{f}\,.
\end{equation}
\end{proof}  

\begin{lemma} \label{lem:Qsubsetminus}
If $F_1$ and $F_2$ are two families of sets of positive integers, and if $F_1 \subset F_2$\,,
\begin{equation} \label{eq:Qsubsetminus}
\frac{Q\left(f,\,F_2\right)}{Q\left(f,\,F_1\right)} = Q\left(f,\,F_2 \setminus F_1\right)\,.
\end{equation}
\end{lemma}
\begin{proof}  This follows from Lemma~\ref{lem:Qunion}.
\begin{equation}
\frac{Q(f, F_2)}{Q(f ,F_1)}\ =\ \frac{Q(f, F_1) \cdot Q(f, F2 \setminus F_1)}{Q(f, F_1)}\ =\ Q(f, F_2 \setminus F_1)\,.
\end{equation}
\end{proof}  

\begin{definition}  Let $ F_{\leq k}$ be the family of all finite subsets of the first $k$ positive integers, excluding $\emptyset$.  Let $j \in\{1, 2,\ldots, k\}  $, and let $ F_{\leq k \ni j}$ be the subset of  $ F_{\leq k}$ consisting of sets that contain the integer $j$.   And let $ F_{\leq k,\backslash j}$ be the subset of $ F_{\leq k}$ consisting of sets that do not contain $j$:
\begin{align}
 F_{\leq k} & = \{ S\,|\, S \subset \{1, 2, \ldots, k\}\,, \\
 F_{\leq k \ni j} & = \{ S\,|\, S \subset \{1, 2, \ldots, k\},\ j \in \{1, 2, \ldots, k\},\  j \in  S\}\,, \\
 F_{\leq k \backslash  j} & = \{ S\,|\, S \subset \{1, 2, \ldots, k\},\ j \in \{1, 2, \ldots, k\},\ j \notin  S\}\,.
\end{align}
\end{definition}

\begin{definition}
The family of nonempty subsets of the positive integers is
\begin{equation*}
F_\mathbb{N}^* = \{S|S \subset \mathbb{N},\ S \neq \emptyset \}.
\end{equation*}
\end{definition}

\begin{definition}
The family of nonempty subsets of the first $k$ positive integers is
\begin{equation*}
F_{\leq k}^* = \{S|S\subset \{1,\ldots ,k\},\ S \neq \emptyset \}
\end{equation*}
\end{definition}

\begin{definition}
The family of nonempty subsets of the positive integers whose maximum integer is $m$ is
\begin{equation*}
F_{\max = m}^* = \{S|S \subset \mathbb{N},\ S \neq \emptyset,\ \max(S)=m\}.
\end{equation*}
\end{definition}

\section {Proof of Mutiplicative Function Extension (Theorem~\ref{thm:main})}

\begin{proof}

Suppose that $j,\,k \in \mathbb{N}$ and $j \leq k$.  By Lemma~\ref{lem:Qunion} and the definition of Q, 
\begin{align}
Q\left(f_j,\, F_{\leq k}\right) &=\ Q\left(f_j,\, F_{\leq k \ni j}\right) \cdot Q\left(f_j,\, F_{\leq k \backslash j}\right) \nonumber \\ \\
&=\ \frac{\prod\limits_{ S \in  F_{\leq k \ni j}^{\,\text{odd}}} P(f_j,\, S)}{\prod\limits_{ S \in  F_{\leq k \ni j}^{\,\text{even}}} P(f_j,\, S)} \cdot \frac{\prod\limits_{ S \in  F_{\leq k \backslash j}^{\,\text{odd}}} P(f_j,\, S)}{\prod\limits_{ S \in  F_{\leq k \backslash j}^{\,\text{even}}} P(f_j,\, S)}\,\,. \label{eq:QN}
\end{align}
If a set $ S$ contains the integer $j$, and $j$ is removed, its cardinality changes: from odd to even, or from even to odd.
Therefore, Lemma~\ref{lem:erasek}  gives 
\begin{equation}
\prod\limits_{ S \in  F_{\leq k \ni j}^{\,\text{odd}}} P(f_j,\, S) = \prod\limits_{ S \in  F_{\leq k \backslash j}^{\,\text{even}}} P(f_j,\, S)
\end{equation}
and
\begin{equation}
\prod\limits_{ S \in  F_{\leq k \ni j}^{\,\text{even}}} P(f_j,\, S) = \prod\limits_{ S \in  F_{\leq k \backslash j}^{\,\text{odd}}} P(f_j,\, S)
\end{equation}

This provides compete cancellation in Eq.~\ref{eq:QN}, so that
\begin{equation} \label{eq:Qempty1}
Q\left(f_j,\, F_{\leq k} \right) \,\,=\,\, 1\,,\,\,j\leq k\,.
\end{equation}

\begin{remark}
At this point, we must confront a delicate issue.  The Riemann Rearrangement Theorem says that an infinite series can give different results if the order of addition of its terms is changed, unless the series is absolutely convergent.  Mindful that an infinite product can be converted into an infinite sum by taking its logarithm, we must therefore consider the order in which the factors of Q in Eq.~\ref{eq:Qempty1} are to be multiplied. 
\end{remark} 

When we form $f = \prod\limits_{j=1}^\infty f_j$\,,\ we form each partial product by multiplying the previous partial product by $f_{j+1}$, i.e. by what we obtain from $f_j$ after we increment $j$ by $1$.\ \  Eq.~\ref{eq:Qempty1}, with its requirement $j\leq k$, causes us to ask:  What sets must be added to $F_{\leq k}$ to form $F_{\leq (k+1)}$?  The answer is given by
\begin{equation*}
F_{\leq (k+1)} = F_{\leq k} \cup F_{\max =( k+1)}\,.
\end{equation*} 
By iteration, this leads to the formula
\begin{equation}
F_{\leq k} = \emptyset \cup   \left(\bigcup_{\,m=1}^{\,k}\ F_{\max =m}^*\right)\,.
\end{equation}
We note that, since $| \emptyset |= 0$ is even, $P(f_k,\,\emptyset)$ is in the denominator of Q.  Using Eq.~\ref{eq:Qempty1}, Lemma~\ref{lem:Qunion},  and Lemma~\ref{lem:emptysetf},
\begin{align*}
1 &= Q\left(f_j,\, F_{\leq k} \right) \\
&= Q\left(f_j,\, \emptyset\ \cup \left(\bigcup_{\,m=1}^{\,k}\ F_{\max =m}^*\right)\right) \\
&=  Q(f_j,\, \emptyset )\ Q\left(f_j,\, \bigcup_{\,m=1}^{\,k}\ F_{\max =m}^*\right) \\
&=\frac{\prod\limits_{m=1}^k\ Q(f_j,\,F_{\max = m}^*)}{P(f_j,\,\emptyset)} \\
&= \frac{\prod\limits_{m=1}^k\ Q(f_j,\,F_{\max = m}^*)}{f_j}\,,
\end{align*}
or
\begin{equation} \label{eq:fkQmax}
f_j = \prod_{m=1}^k\ Q(f_j,\,F_{\max = m}^*)\,.
\end{equation}

We want to use Eq.~\ref{eq:fkQmax} to assemble $f = \prod\limits_{j=1}^\infty f_j$\,, but the requirement $j\leq k$ requires us to simultaneously take the limit $k \rightarrow \infty$\,, producing
\begin{align}
f(z) &= \prod_{j=1}^\infty\ f_j \\
&= \prod_{j=1}^\infty\ \prod_{m=1}^\infty\ Q(f_j,\, F_{\mathbb{N},\, max=m}^*) \\
&= \prod_{m=1}^\infty\ Q(f,\, F_{\mathbb{N},\, max=m}^*) \\
&= Q(f,\, F_\mathbb{N}^*)\ _{(GEO)}\,.
\end{align}
\end{proof}

Theorem~\ref{thm:main} is proved.

If we express the theorem in terms of $F_{\mathbb{N}}$ by using Lemma~\ref{lem:emptyset}, we obtain
\begin{equation}
\label{eq:prettiest}
\boxed{ Q\left(f,\, F_{  \mathbb{N}} \right)\ _{(GEO)}\ =\ 1}
\end{equation}

We have boxed Eq.~\ref{eq:prettiest} because this ``universal'' relationship between analytic functions and all finite subsets of $\mathbb{N}$  is perhaps the most concise and most meaningful statement of our function extension method.

\begin{remark}
Nowhere in the proof of Theorem~\ref{thm:main} was it necessary to specify values for the reciprocal geometric ratios $r_k$.
The quotient $Q\big (f(z),\, F_{  \mathbb{N}}^*\big )$ on the righthand side of Eq.~\ref{eq:quotient} is independent of the $r_k$.
As the $r_k$ vary, the function evaluation points compress toward or expand away from the origin, but the quotient of function values is invariant.  We call this property of analytic functions \uline{\textit{elastic invariance}}.
\end{remark}

\section{Aspects of Mutiplicative Function Extension}
\begin{enumerate}
\item The points where the function is evaluated are independent of the function.
\item The function evaluation points are proportional to $z$.  This gives the method a scaling property.  If we take the function evaluation points for calculating $f(z)$, double them, and form the quotient of function products for the doubled points, we obtain $f(2z)$, the value of the function twice as far from the origin.
\item The requirement $|r_k|>1$ means that the geometric ratio $1/|r_k| <1$, assuring convergence of the $k^{th}$ geometric sequence.  We require $f(0)=1$ for infinite product convergence ~\cite{Remmert}.
\item To calculate $f(z)$,\,Theorem~\ref{thm:main} multiplies or divides (when $|S|$  is odd or even, respectively) values of the function at points that are  no further from the origin than $z$ and converge to the origin.\ \ The theorem continues a function ``outward'' from geometric sequence points nearer the origin, because as $n_k \rightarrow \infty$\, the argument of the function approaches
\begin{equation*}
\lim_{n_k \rightarrow \infty}\left( \prod_{\substack{(\forall k \in  S)}}\ \frac{[(r_k)^k -1]^{1/k}}{(r_k)^{n_k}}\right) \cdot z\  = \ 0.
\end{equation*}
If $r$ is real, then every function evaluation point is closer to the origin than $z$ because
\begin{equation*}
\left( \prod_{\substack{(\forall k \in  S)}}\ \frac{[(r_k)^k -1]^{1/k}}{(r_k)^{n_k}}\right) \cdot z\  \ <\ \  \left(\prod_{\substack{(\forall k \in  S)}}\ \frac{[(r_k)^k]^{1/k}}{(r_k)^{n_k}}\right) \cdot z\ \ =\ \ \left(\prod_{\substack{(\forall k \in  S)}}\ \frac{1}{(r_k)^{n_k-1}}\right) \cdot z\ \ < \ \ z\,,
\end{equation*}
since $r_k > 1$ and $n_k \geq 1$.
\item By restriction, the theorem also applies to curves $\mathbb{R} \rightarrow \mathbb{C}$.  The theorem can extend a segment of an analytic curve $\gamma : \mathbb{R} \rightarrow \mathbb{C}$ by \uline{multiplying and dividing together points on the curve}, so long as the segment includes the point $\gamma(0)=1$.   To continue a function from an arbitrary base point $z_0$ where $f(z_0)\neq 1$, one would apply our method to the function $f(z)/f(z_0)$ if $f(z_0)\neq 0$, or to $1+f(z)$ if $f(z_0)=0$.

\item We conjecture that our method could also extend a smooth curve through a Lie group, since elements of a Lie group can also be multiplied and ``divided'' (i.e., multiplied by inverse elements)~\cite{Lee}.

\item The requirement  that  $\Re\left((r_k)^k\right) \geq 1/2$, ensures that, when $r$ is complex, $(r_k)^k-1$ has a real part that is less than or equal to the real part of $(r_k)^k$, so that every function evaluation point is no further from the origin than $z$, i.e., lies where $f$ is assumed to be analytic.  This requirement is not necessary if the function $f$ is analytic everywhere in $\mathbb{C}$.
\item If a function is Taylor-Analytic, it is Multiplicatively-Smooth (i.e., satisfies Theorem~\ref{thm:main}).  But not every Multiplicatively-Smooth function is Taylor-Analytic.
Multiplicative-Smoothness for $f$ requires Taylor-Analyticity for $\ln (f)$; however this does not imply Tayor-Analyticity for $f$ itself. The function $f(x)=1-\exp(-1/x^2)$, with $f(0)$ defined to be 1, is Multiplicatively-Smooth.  But it is is not Taylor-Analytic at the origin because all of its derivatives are zero at $z=0$~\cite{Randol}.  Its Taylor-MacLauren Series calculates that $f(z)=0$ everywhere, which is of course not true.  In this sense, ``has elastic invariance'' might serve as a more inclusive definition of function smoothness than ``has a power series''.
\end{enumerate}

\section{Convergence}
We show that the product of infinite products on the righthand side of Eq.~\ref{eq:quotient} converges for all functions $f$ that satisfy the requirements of Theorem~\ref{thm:main}.

We use the fact that the product on the righthand side of equation Eq.~\ref{eq:quotient} is invariant under changes to the $r_k$.  In particular, it is invariant if we take the limit $r_k\rightarrow \infty$.  We note that
\begin{align*}
\lim\limits_{r_k \rightarrow \infty} \frac{\left[(r_k)^k - 1)\right]^{1/r_k}}{(r_k)^{n_k}}\  &=\  \lim\limits_{r_k \rightarrow \infty}\   \frac{r_k}{(r_k)^{n_k}}\\
   &=
\begin{cases}
1 & \text{if}\ n_k = 1,\\
0 & \text{if}\ n_k > 1.
\end{cases}
\end{align*}
When $r_k \rightarrow \infty$ every factor in the righthand side of Eq.~\ref{eq:quotient} that has $n_k > 1$ contributes a factor of $1$ to the product, because $f(0)=1$.  And every factor on the righthand side of Eq.~\ref{eq:quotient} that has $n_k = 1$ (i.e., $\forall k \in S$) contributes a factor of $f(z)$.

As illustrated in Eq.~\ref{eq:explicit},  in $ F_{\mathbb{N},\, max=m}^*$ with $m > 2$ the number of subsets with $|S|$ odd is equal to the number of subsets with $|S|$ even.  When $m=1$ there is only one subset, namely $\{1\}$. 
  This means that when $r_k\rightarrow \infty\ \forall k$ the righthand side of Eq.~\ref{eq:quotient} reduces to
\begin{align*}
&  \prod_{m=1}^\infty \,\,\  \prod_{S \in  F_{\mathbb{N},\, max=m}^*} \,\lim\limits_{r_k\rightarrow \infty, \forall k} \,\,  \prod_{\substack{(\forall k \in  S) \\
n_k = 1
}}^\infty\  \left[ f\left(\left[\prod\limits_{k\in S}\frac{[(r_k)^k-1]^{1/k}}{(r_k)^{n_k}}\right] \cdot z \right)\right]^{\big{(-1) ^{\,| S|-1 }}} \\
=\ \ \ &f(z) \times \frac{f(z)}{f(z)} \times \frac{f(z) f(z)}{f(z) f(z)} \times \frac{f(z) f(z) f(z) f(z)}{f(z) f(z) f(z) f(z)} \times \cdots \\
=\ \ \ &f(z)
\end{align*}
Since the righthand side of Eq.~\ref{eq:quotient} is independent of the $r_k$, we have shown that Eq.~\ref{eq:quotient} is valid for all functions $f$ that meet the requirements of Theorem~\ref{thm:main}, and that its value is $f(z)$, not $\pm \infty$.  This infinite product of infinite products converges.

\section{Truncated Multiplicative Function Extensions and Other Results}
\subsection{A Simplification}
The following Corollary simplifies practical calculations of Mutiplicative Function Extensions.
\begin{corollary}
If the same geometric ratio is used for every positive integer, $\forall k\ \, r_k =r$, in a Multiplicative Function Extension, then
\begin{equation}
\label{eq:factors}
P(f,\, S,\,z) \,\,=\,\,\,\,\prod_{n=| S|}^\infty\,\left[\,f\left(\,\left[\prod_{k\in S}\,(r^k-1)^{1/k}\right]\frac{z}{r^n}\right)\,\right]^{\,\,\textrm{\large{$\binom{n-1}{| S|-1}$}}}\,.
\end{equation}
\end{corollary}
\begin{proof}
This result consolidates all $| S|$ exponents $n_k$, for $k \in S$, into a single integer $n$.  It recognizes that $\prod\limits_{k \in S}1/r^{n_k} = 1/r^{\left(\sum\limits_{k \in S} n_k \right)} = 1/r^n$.  The binomial-coefficient exponent for each $f$ value is the number of ways in which $| S|$ positive integers can add up to $n$\ \ \cite{MacMahon}; it is the multiplicity of the factor of $f$ for the sequence point that has $n$ factors of $r$ in its denominator.  The index $n$ starts at  $|S|$ because each of the $|S|$ exponents $n_k$ starts at $1$.
\end{proof}
The function $f$ is evaluated on a single geometric sequence for each set of positive integers $S$.
This reduces the number of points at which the function must be evaluated.  It improves the performance of real-time engineering applications, in which each factor to be multiplied is obtained by measuring the magnitude of a smooth signal, rather than evaluating a smooth function.  One such application, predicting the future value of a smooth signal from previous geometric samples, is described in Ref.~\cite{Miller}.
\subsection{Truncation of Multiplicative Function Extensions}
In practice, we can only perform finite calculations.  This requires us to truncate function extensions that are infinite sums or products.  Multiplicative Function Extensions require two truncations.
\begin{enumerate}
\item We limit the number of products $P$, by limiting the integer subsets used.  We use the subsets of the first K positive integers, $S \subset S_{max} = \{1, 2, \ldots, K\}$.
\item We limit the number of points used from each geometric sequence.  With a common $r$-value, we limit the consolidated exponent $n$ to a maximum value of N for each integer subset.
\end{enumerate}
We can then specify a Multiplicative Function Extension truncation as a triplet of values:  $(S_{max}, N, r)$.

\begin{figure}[h!]
\begin{center}
\subfigure{
\includegraphics[scale=.23]{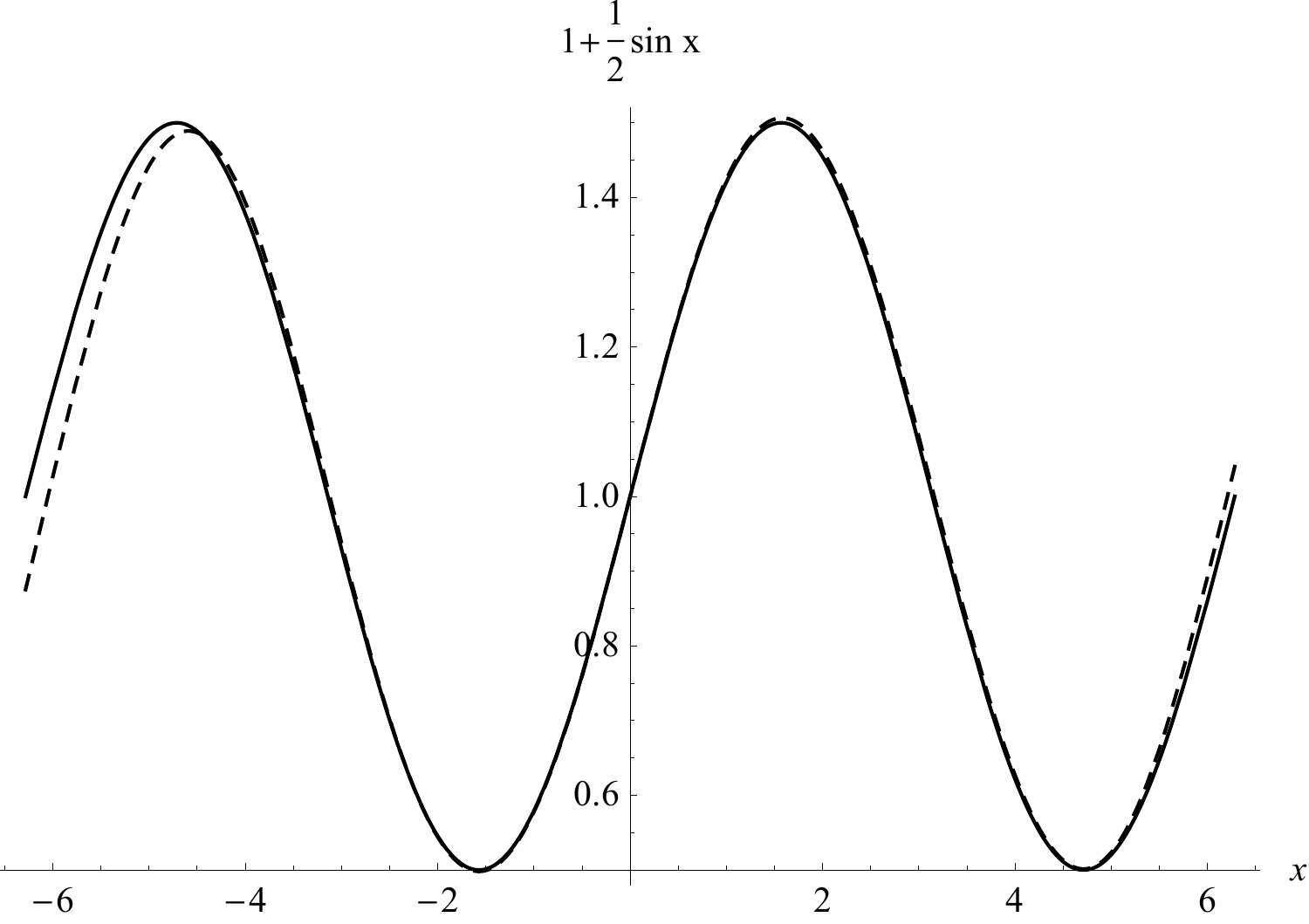}}
\subfigure{
\includegraphics[scale=.23]{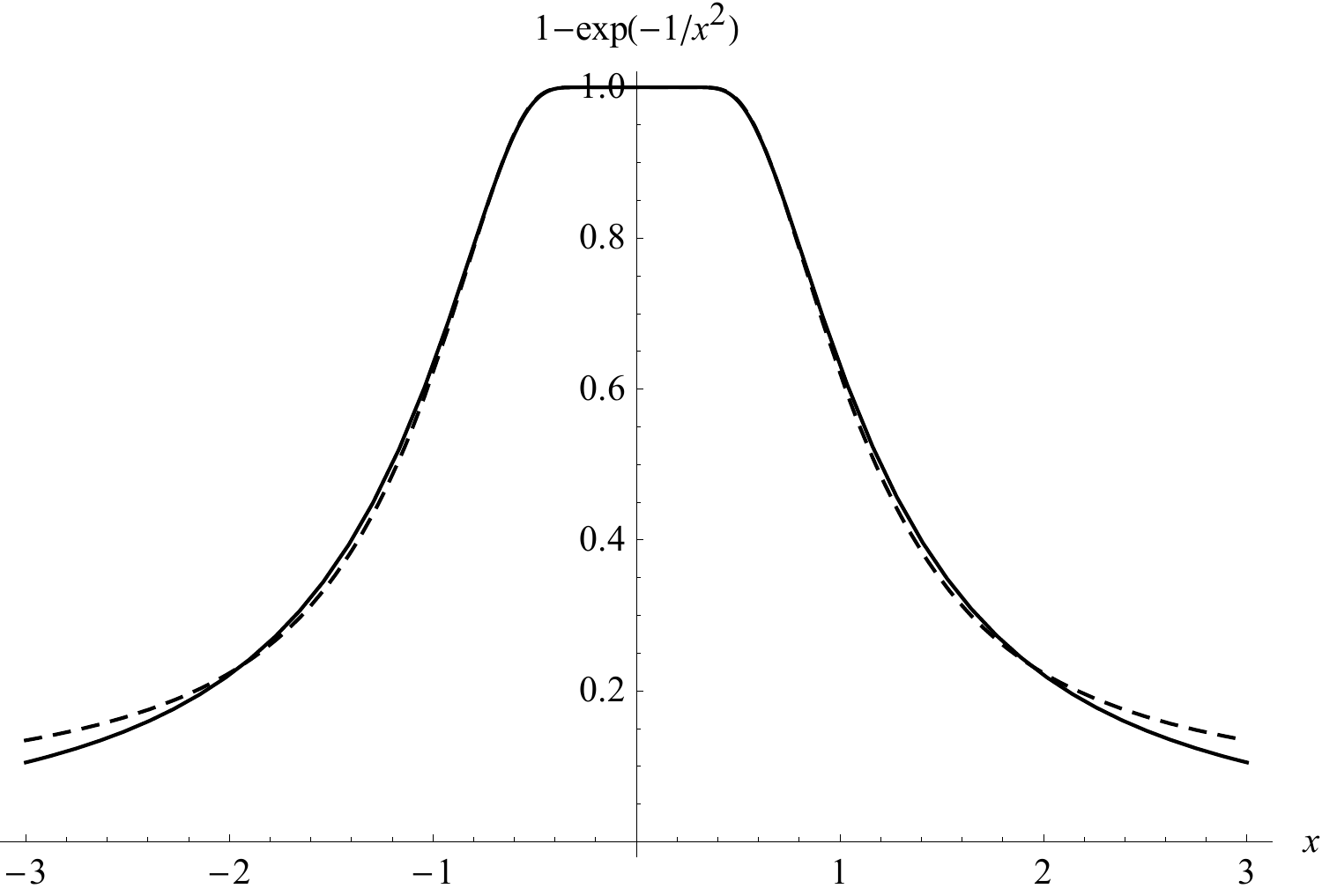}}
\caption {Truncated Multiplicative Function Extensions for (a) $1+\frac{1}{2}\sin(x)$, with $S_{max} = \{1,2,3,4\},\, N=40, \text{ and } r=2$; and (b)~$1~-~e^{-1/{x^2}}$, with $S_{max}=\{2,4,6,8\},\, N=20,\, \text{and } r=2$, based at zero. Solid lines are the truncations; dashed lines are the functions.}
\label{halfsin-nanal}
\end{center}
\end{figure}

\subsection{Graphs of Truncations for Two Elementary Real Functions}
Figure \ref{halfsin-nanal}(a) shows a $(\{1,2,3,4\},40,2)$ Multiplicative Function Extension for the function $1+\frac{1}{2}\sin(x)$. Increasing the cutoff $S_{\,\text{max}}$ gives more oscillations to this graph. As with cutoff MacLauren Series, the further away that the target point $z$ is from the base point (the origin), the greater the error introduced by the cutoff.

Figure \ref{halfsin-nanal}(b) illustrates a $(\{2,4,6,8\},20,2)$ Multiplicative Function Extension for $1-~e^{-1/x^2}$. As Eq.~\ref{eq:taylog} makes clear, an even function only has even-order components, and therefore only requires even integers in the sets that generate its geometric sequences.

\subsection{Truncation Error}
The following theorem quantifies the error introduced by truncation.  We omit its proof, which is lengthy and tedious.
\begin{theorem} (Truncation Error)
\begin{align} \label{eq:remain}
f(z) \,\,\approx\,\, Q_N&\left(f,\,F_{\leq K}^*\right)\, \cdot\, Q\left(f\left(\frac{z}{r^N}\right),\,F_{\leq K}^*\right)\, \cdot\, Q\left(f_{>K},\,F_{\, >K}^*\right)\,, \\
where\ \ Q_N\left(f,\,F\right) \,\,&:=\,\, \frac{\prod\limits_{S \in F,\,\,|S|\, \textrm{is odd}}\,P_N\left(f,\,S\right)}{\prod\limits_{S \in F,\,\,|S|\, \textrm{is even}}\,P_N\left(f,\,S\right)}\,, \nonumber \\
P_{N}(f,\,S) \,\,&:=\,\, \prod_{\ldots,\,n_j,\,\ldots\, =\, 1}^{[\,\,N/|S|\,\, ]}\,f\left(\,\left[\,\prod_{j\,\in\,S}\,\frac{(r^j-1)^{1/j}}{r^{n_j}}\right]\,z\right)\,, \nonumber \\
F_{\leq K}^*\,\, &:=\,\, \{S\,|\,S \subset \{1, 2, \ldots, K\},\,S\neq \emptyset,  \}\,, \nonumber \\
F_{>K}^*\,\, &:=\,\, \{S\,|\,S \subset \{K+1, K+2, \ldots, \infty\},\,S\neq \emptyset, \,\,|S|<\infty \}\,, \nonumber \\
and\ \ f_{> K} \,\,&:=\,\, \prod_{j=K+1}^\infty\, f_j\,. \nonumber
\end{align}
\end{theorem}
\begin{remark} \ \ \\
\begin{itemize}
\item $Q\left(f_{>K},\,F_{\, >K}^*\right)$ is the error-factor from truncating the subsets of the positive integers (i.e.number of geometric sequences) to the subsets of the first K positive integers.  The error factor is one (indicating no error) if the function $f$ has no function factors of higher order than $K$.
\item $Q\big (f\left(z/r^N\right),\,F_{\leq K}^*\big )$ is the error-factor from subsequently truncating the number of points per geometric sequence to N.  The error factor approaches one (indicating no error) when $N \rightarrow \infty$, because then $f(z/r^N)) \rightarrow f(0) =1$.
\item $Q_N\left(f,\,F_{\leq K}^*\right)$ is the value of the Multiplicative Function Extension after both truncations.
\item The theorem is exact, not approximate, if $N$ is divisible by $|S|$ for all $S \subset F_{\leq K}^*$.
\end{itemize}
\end{remark}

\subsection{Other Results}
We present several other results without proof.

If we sample an arbitrary component $f_j$ on the $k^{th}$ singlet sequence, and let $r$ approach 1 from above, $r\downarrow 1$, we get a surprisingly simple result.
\begin{proposition}[Function Factor Cutoff Tool] 
Suppose $f(z) = \prod_{k=1}^\infty f_k(z) = \prod_{k=1}^\infty \exp(c_k z^k)$.  Then
\begin{align}
\lim_{r \downarrow 1}\,P(f_j,\,\{k\}) &\,\,=\,\, \label{eq:kcut}
\begin{cases}
f_k & \textrm{if $j = k$}\,,\\ 
1 & \textrm{if $j>k$ \textrm{or} $c_j = 0$}\,, \\
0 & \textrm{if $j<k$ \textrm{and} $c_j<0$}\,, \\
\infty & \textrm{if $j<k$ \textrm{and} $c_j>0$} .
\end{cases}
\end{align}
\end{proposition}

The $j>k$ case of Eq.~\ref{eq:kcut} allows us to calculate a single component $f_k$, by nullifying the effect of lower-order samplings of $f_j$.
\begin{theorem}[Isolating One Function Component Factor]
The $k^{th}$ functional component factor of $f$ is given by:
\begin{equation}
f_k = \lim\limits_{r\downarrow 1} \,\,Q\left(f,\, F_{\mathbb{N},\, max=k}^*\right)\,.  \label{eq:fk}
\end{equation}
\end{theorem}

Using this result to isolate $f_1$, we can calculate the derivative of a function as the limit of an infinite product.
\begin{theorem}[Product Derivative]
 If $f\text{:\ }  \mathbb{C}\rightarrow\mathbb{C}$ is analytic and nonzero in an open star-shaped neighborhood $\mathbb{U}$ of $z$, then
\begin{equation}
\label{eq:prodderiv}
f'(z) \,= \,\, \frac{f(z)}{\Delta z} \cdot\, \lim\limits_{r \downarrow 1} \,\,\log  \,\prod_{n=1}\limits^\infty \left[ f\left(z + \frac{r-1}{r^n} \, \Delta z\right) / f(z) \right]
\end{equation}
for any nonzero $\Delta z$ with $z + \Delta z \in\mathbb{U}$.
\end{theorem}
We do \uline{not} need to take the limit $\Delta z \rightarrow 0$; taking the limit $r \downarrow 1$ suffices.  It is easy to show that the product derivative obeys the Leibniz product rule.

\section{Prime Functional Identity}
Multiplicative Function Extension relates to number theory when each geometric ratio is chosen to be a common power of the reciprocal of a prime number.

\textbf{Definitions}\ \ In what follows, we use the following definitions.
\begin{itemize}
\item  $P(n)$ is the set of prime factors of $n$,
\item $\omega(n)$ is the number of distinct prime factors of $n$,
\item $\pi (p)$ is the number of primes less than or equal to $p$, so that $\pi(p_k) = k$ if $p_k$ is the $k^{\mathrm{th}}$ prime,
\end{itemize}
\begin{definition} \label{def:GPO}
\uline{Greatest Prime Order} (GPO) is an ordering of groups of the positive integers in which
\begin{enumerate}
\item $1$ is the first integer (if present), because it has no prime factors.
\item The next group of integers is  $\{2^{n_1}\}_{n_1=1}^\infty$
\item The next group of integers is  $\{2^{n_1} \cdot 3^{n_2}\}_{n_2=1}^\infty |_{n1=0}^\infty$
\item The next group of integers is  $\{2^{n_1} \cdot 3^{n_2} \cdot 5^{n_3}\}_{n_3=1}^\infty |_{n_2,n_1=0}^\infty$
\item etc.
\end{enumerate}
\end{definition}
We associate $n_1$ with the prime $2$ because $2$ is the $1^{st}$ prime, etc.

\begin{remark}  Greatest Prime Order partitions the positive integers into equivalence classes with the equivalence relation ``have the same greatest prime factor".\ \ GPO orders these equivalence classes, but it does not specify the order of a particular integer within its equivalence class.
\end{remark}

\begin{theorem}
\label{thm:second}
Suppose that $U \subset \mathbb{C}$ is an open disk-shaped neighborhood of the origin and that $z \in U$.  Suppose that the function $f:\mathbb{C} \rightarrow \mathbb{C}$ is analytic and nonzero in $U$ and that $f(0)=1$.\, Let $s \in \mathbb{C}$\, with \ $|2^s| > 1$ and \ $\Re(2^s) \geq 1/2 $.\ \ Let $P(n)$ be the set of prime factors of $n$,\,\ $\omega(n)$ be the number of distinct prime factors of $n$,
$\pi (p)$ be the number of primes less than or equal to $p$, and let $p_m$ be the $m^{th}$ prime so that $m = \pi(p_m)$.\ \ Then
\begin{equation}
f(z)\ \  =\ \ \prod_{m=1}^\infty\ \ 
  \prod_{\substack{n\, \in\, \mathbb{N} \\
\max(P(n))=p_m
}}
\ \left[f\left(\frac{\prod\limits_{p\, \in \,P (n)}\left[p^{\,\,s \cdot \pi (p)} \,\,-1\right]^{1/{\pi (p)}}}{n\,^s}\,\cdot\, z \right)\right]^{\big {(-1)} ^{\,\omega (n) -1 }}\,.
\label{eq:primef}
\end{equation}
\end{theorem}
\ \newline
\begin{proof}
We choose the reciprocal geometric ratios in Theorem~\ref{thm:main} to be all successive prime numbers raised to a common power:  $r_k=(p_k)^s$\,, where $p_k$ is the k-th prime.  Then we use the fundamental theorem of arithmetic:  for every positive integer $n$ greater than one,\ \ $1/n = 1/ \prod_{k \in S}\, (p_k)^{n_k}$ for some unique $S \subset \mathbb{N}$ and for some unique tuple of exponents $\{n_k\}_{k \in S}$.\  Therefore,
\begin{equation*}
1/\prod_{k \in S}\left(r_k\right)^{n_k} =\ 1/\prod_{k \in S}\left[(p_k)^s\right]^{n_k} =\ 1/\prod_{k \in S}\left[(p_k)^{n_k}\right]^s =\ 1/n^s.
\end{equation*}
In Theorem~\ref{thm:main}, we substitute $S=P (n)$\,.\ \ Then \, $p_k \in P (n)$\, when $k \in S$,\, because $k=\pi (p_k)$\,.

The integer $m$ labels the equivalence classes of Greatest Prime Order for the integers $\{2,3,\ldots,\infty\}$.  The first two products ensure that the integers $n$ appear in Greatest Prime Order.  The index $m$ starts at $2$ because the first prime is  $r_1=2$.

Since the subsets $\{S\,|\,S \subset F_{\mathbb{N}}^*\}$ appear in Greatest Element Order in Theorem~\ref{thm:main},
the combination $n=\prod\limits_{k \in S}{r_k}^{n_k}$ produces the integers in Greatest Prime Order. Greatest $k$ (greatest element) order for $S$ means greatest $r_k$ (greatest prime) order for $n$.
\end{proof}
We see that Theorem~\ref{thm:second} has the same denominator $n^s$ as the Riemann zeta function \cite{Edwards}.

\begin{definition}
We call Eq.~\ref{eq:primef} the \uline{prime functional identity}.
\end{definition}

\begin{definition} \label{def:defmus}
Let $n \in \mathbb{N}$ and $s \in \mathbb{C}$.  We define the \uline{generalized M\"obius function} $\mu^{\star} (n,s)$ to be
\begin{equation*} \label{eq:defmus}
\mu^{\star} (n,s)\, \  =\, \ \frac{ {(-1)} ^{\omega (n) } \cdot \prod\limits_{p\, \in P (n)}\,\left[\, p^{\,\,s \cdot \pi (p)} \,\,-1 \right]^{1/\pi (p)}}{n\,^s}\,,
\end{equation*}
where we define $\mu^{\star}(1,\,s) = 1$.
\end{definition}

Since $1$ has no prime factors, $\omega(1)=0$, and therefore consistency requires the first value to be $\mu^{\star} (1,\,s)=(-1)^0/1^s=1$.

Definition~\ref{def:defmus} is motivated by the following definition and corollary.
\begin{definition}
$\mu(n)$ is the \uline{M\"obius function}:
\begin{align*}
\mu(n) = 
\begin{cases}
(-1)^{\omega(n)} & \text{if n is not divisible by the square of any prime number,}\\
0 & \text{otherwise.}
\end{cases}
\end{align*}
\end{definition}

\begin{corollary} \label{cor:mumus}
\begin{equation}
\lim_{Re(s)\rightarrow \infty} \mu^{\star} (n,s) = \mu (n),
\end{equation}
\end{corollary}

\begin{proof}
Let $\Omega_p (n)$ be the greatest number of factors of the prime number $p$ that divide $n$, i.e., the multiplicity of the prime $p$ in the prime factorization of $n$.  Then by the fundamental theorem of arithmetic
\begin{equation*}
n = \prod_{p\, \in P (n)} p^{\Omega_p (n)}.
\end{equation*}
\begin{align*}
\textrm{Therefore}\ \lim_{\Re(s) \rightarrow \infty} \left[\, p^{\,\,s \cdot \pi (p)} \,\,-1 \right]^{1/\pi (p)} &\approx\,\  \left[\, p^{\,\,s \cdot \pi (p)} \right]^{1/\pi (p)} \\
&\approx\,\ p^s.
\end{align*}
Therefore
\begin{align*}
\lim_{\Re(s) \rightarrow \infty}\ \mu^{\star} (n,s)\ &=\ \lim_{\Re(s) \rightarrow \infty}\frac{ {(-1)} ^{\omega (n) } \cdot \prod\limits_{p\, \in P (n)}\,\left[\, p^{\,\,s \cdot \pi (p)} \,\,-1 \right]^{1/\pi (p)}}{n\,^s} \\
\\
&= (-1)^{\omega(n)} \lim_{\Re(s) \rightarrow \infty}\  \prod_{p\, \in P (n)}\,\ \left(\frac{p}{ p^{\Omega_p (n)}}\right)^s \\
&=
\begin{cases}
(-1)^{\omega(n)} & \text{if }\Omega_p(n) = 1 \ \ \text{for every\ } p\, \in P (n) \\
0 & \text{otherwise.}
\end{cases} \\
&= \mu(n)
\end{align*}
\end{proof}

\begin{corollary}
If $\Re(2^s) \geq 1/2$, then\ \ $\ |\mu^{\star} (n,s)| \leq 1$\,.
\end{corollary}
\begin{proof}
\begin{equation*}
|\mu(n,s)| = \left(\prod_{p \in P(n)} |( p^{\,s\cdot \pi(p)} - 1)^{1/\pi(p)} | \right) /|n^s|\ \leq \left(\prod_{p \in P(n)} | p^s|\right)/|n^s| \leq 1\, .
\end{equation*}
\end{proof}

\section{Prime Number Identities}
For each function $f$, Theorem~\ref{thm:second} produces a prime number identity. An example:

\begin{theorem}  For every $s \in \mathbb{C}$,
\begin{equation}
\label{eq:primer}
\sum_
{\substack{\ \\
n \in \mathbb{N} \\
(GPO)
}}
\mu^{\star} (n,s) = 0,
\end{equation}
where $\mu^\star (n,s)$ is the generalized M\"obius function defined in Definition~\ref{def:defmus} and (GPO) indicates Greatest Prime Order for the positive integers, defined in Definition~\ref{def:GPO}.
\end{theorem}

\begin{proof}
We choose $f(z) = e^z$, which is analytic everywhere in the complex plane.  Then Eq.~\ref{eq:primef} becomes
\begin{align*}
e^z\ \ & =\ \ \prod_
{\substack{\ \\
n \ \geq 2 \\
(GPO)
}}
\ \left[\exp\left(\frac{\prod\limits_{p\, \in \,P (n)}\left[p^{\,\,s \cdot \pi (p)} \,\,-1\right]^{1/{\pi (p)}}}{n\,^s}\,\cdot\, z \right)\right]^{\big {(-1)} ^{\,\omega (n) -1 }}\, \\
&=\ \ \exp\left( \sum_
{\substack{\ \\
n \ \geq 2 \\
(GPO)
}}
 \ \frac{{(-1)} ^{\,\omega (n) -1 }\ \cdot \prod\limits_{p\, \in \,P (n)}\left[p^{\,\,s \cdot \pi (p)} \,\,-1\right]^{1/{\pi (p)}}}{n\,^s}\,\cdot\, z \right)\,.
\end{align*}

We take logarithms of both sides, and divide by $z$. Then we move all right-hand terms to the left side of the equation, by changing each sign from $(-1)^{\omega (n)-1}$ to $(-1)^{\omega (n)}$,\, to obtain 
\begin{equation}
1\,\  +\,\  \displaystyle\sum_
{\substack{\ \\
n \ \geq 2 \\
(GPO)
}}
\ \  \frac{ {(-1)} ^{\omega (n) }\ \cdot \prod\limits_{p\, \in P (n)}\,\left[\, p^{\,\,s \cdot \pi (p)} \,\,-1 \right]^{1/\pi (p)}}{n\,^s}\, \ =\, \ 0\,. 
\end{equation}
Since $\mu^\star(1,\,s)=1$, we consult Definition~\ref{def:defmus}, and see that this is equivalent to Eq.~\ref{eq:primer}.
\end{proof}

\begin{corollary}
\label{cor:muzero}
\begin{equation}
\sum_{\substack{
n \in \mathbb{N} \\
(GPO)
}}
 \mu(n) = 0.
\end{equation}
\end{corollary}
\begin{proof}
Eq.~\ref{eq:primer}  is an identity for all $s$; we take the limit $\Re(s) \rightarrow \infty$, and use Corollary~\ref{cor:mumus}.
\end{proof}
We hope our techniques illuminate the relationship between function smoothness and the prime numbers.
%%                                          %%
%% Backmatter begins here                   %%
%%  After submission of the .TEX file,                     %%
%%  you will be prompted to submit your .BBL file.         %%
% if your bibliography is in bibtex format, use those commands:
%\bibliographystyle{bmc-mathphys} % Style BST file (bmc-mathphys, vancouver, spbasic).
%\bibliography{bmc_article}      % Bibliography file (usually '*.bib' )
% or include bibliography directly:

\
\end{document}